\documentclass{article}
\usepackage{amsgen, amsmath, amsfonts, amsthm, amssymb, enumerate,amscd,graphics}
\usepackage[all]{xy}

\newtheorem{defn}{Definition}[section]

\newtheorem{lem}[defn]{Lemma}
\newtheorem{thm}[defn]{Theorem}
\newtheorem{cor}[defn]{Corollary}
\newtheorem{rem}[defn]{Remark}
\newtheorem {conj}[defn]{Conjecture}

\newcommand {\BL}{{\rm L}}

\newcommand {\ZZ}{{\mathbb Z}}

\newcommand {\XX}{{\mathcal X}}
\newcommand {\K}{{\rm K}}
\newcommand {\E}{{\mathcal E}}

\newcommand {\A}{{\mathcal A}}
\newcommand {\AF}{{\mathbb A}}

\newcommand {\CC}{{\mathcal C}}
\newcommand {\Q}{{\mathbb Q}}
\newcommand {\R}{{\mathbb R}}
\newcommand {\OO}{{\mathcal O}}

\newcommand {\HH}{{\mathcal  H}}
\newcommand {\LL}{{\mathcal L}}

\newcommand {\M}{{\mathcal M}}

\newcommand {\V}{{\bf  V}}

\newcommand {\CP}{{\mathbb P}}

\newcommand {\D}{{\mathcal D}}

\newcommand {\TCC}{\tilde{\mathcal C}}

\newcommand {\LD}{{{\mathcal L}_{\Delta}}}
\newcommand {\QD}{Q^{\Delta}}
\newcommand {\TK}{{\tilde{K}}}
\newcommand{\TC}{\tilde{C}}
\newcommand {\TQ}{{\tilde{Q}^{\Delta}}}

\newcommand {\TP}{{\tilde{P}}}

\newcommand {\TQQ}{{\tilde{\mathcal Q}}}
\newcommand {\TFF}{{\tilde{\mathcal F}}}
\newcommand {\FF}{{\mathcal F}}

\def\pic{\operatorname{\rho}}
\def\Pic{\operatorname{Pic}}
\def\Ker{\operatorname{Ker}}
\def\End{\operatorname{End}}

\def\div{\operatorname{div}}
\def\deg{\operatorname{deg}}

\def\Spec{\operatorname{Spec}}

\def\ord{\operatorname{ord}}

\title{Higher Chow cycles on Abelian surfaces and a non-Archimedean analogue of the  Hodge-$\D$-conjecture }
\author{Ramesh Sreekantan}
\begin{document}

\maketitle
\begin{abstract}

 We construct new indecomposable elements in the higher Chow group $CH^2(A,1)$ of a principally polarized Abelian surface over a  $p$-adic local field, which generalize an element constructed by Collino \cite{coll}. These elements are constructed using a generalization, due to Birkenhake and Wilhelm \cite{biwi}, of a classical construction of Humbert. They can be used to prove a non-Archimedean analogue of the Hodge-$\D$-conjecture - namely, the surjectivity of the boundary map in the localization sequence - in the case when the Abelian surface has good and ordinary reduction.

\end{abstract}


\section{Introduction}

The aim of this paper is to prove a non-Archimedean analogue of the Hodge-$\D$-conjecture for Abelian surfaces. This conjecture asserts that the boundary map in the localization sequence of higher Chow groups is surjective. If an $S$-integral version of the Beilinson conjectures were known this would be a consequence of them -- but since they seem a little out of reach at the moment it is of interest to prove this weaker statement. 

The conjecture is the following -- 

\begin{conj} Let $X$ be a projective scheme over a global field  $\K$. Let $p$ be a prime in $\OO_{\K}$. We think of $X$ as a variety over $\K_p$, the completion at $p$. Let $\XX$ be a model over the ring of integers $\OO_{\K_p}$  of $\K_p$  and let $\XX_p$ be the special fibre. We assume $\XX_p$ is smooth - that is - $X$ has good reduction at $p$. 

For $m,n \geq 1$ let 
$$\Sigma^{m,n}_{X}:=\Ker\{CH^{m}(\XX,n-1) \longrightarrow CH^{m}(X,n-1)\}.$$
Then the map --
$$CH^m(X,n)\otimes \Q \stackrel{\partial}{\longrightarrow} CH^{m-1}(\XX_p,n-1)\otimes \Q$$
is surjective and $\Sigma_X^{m,n}$ is finite.
\end{conj}

In fact, one can formulate this conjecture more generally  for  primes $p$  of semi-stable reduction \cite{sree2} - but in this paper we will only deal with primes of good reduction. We can also consider this conjecture for $X$ over a global field itself - looking at the boundary map for {\em all} primes - but that is a much harder question. 

In the local case the surjectivity of $\partial$ alone is sufficient to show $\Sigma_X$ is finite, but  in the global case it only shows that $\Sigma_X$ is torsion.  

 A few  cases of this conjecture, both the local and global version, are known --

\begin{itemize}

 \item  $m=n=1$ and $X=Spec(\K)$ where $\K$ is a global or local field. Here  $\XX=\Spec(\OO_{\K})$ is the ring of integers, $CH^1(X,1)=\K^*$  and $CH^0(\XX_p,0)=CH^0(\XX_p) \simeq \ZZ$. The conjecture is trivially true in the local case and follows from the finiteness of class number in the global case. In fact, the group $\Sigma_X^{1,1}$ is the class group.

 \item $m=k, n=2k-1$, $k > 1$ and $X=Spec(\K)$ where $\K$ is a global or local field. This is an immediate consequence of the results of Quillen on the $K$-theory finite fields - as he showed  that the higher  $K$-groups of finite fields are finite.

 \item  $m=2$, $n=1$ and $X=E \times E$ where $E$ is an elliptic curve over $\Q$ and $p$ a prime of good reduction. Mildenhall \cite{mild} and Flach \cite{flac} independently showed  that this map is surjective in this case. Mildenhall further showed that $\Sigma^{2,1}_X$ is finite  when $E$ is a CM elliptic curve over $\Q$. Finiteness of $\Sigma^{2,1}_X$ in general is still not known. 

 \item  $m=2$, $n=1$ and $X=E_1 \times E_2$ where $E_1$ and $E_2$  are elliptic curves over a local field $\K_p$ with good reduction at $p$. Spiess \cite{spie} showed that this map is surjective in this case.

 \item  $m=2$, $n=1$ and $X=E_1 \times E_2$ where $E_1$ and $E_2$  are elliptic curves over a local field $\K_p$ with bad semi-stable  reduction at $p$. In \cite{sree2} we generalized Spiess' work to the case of semi-stable reduction. Here we had to use a semi-stable model $\XX$ of $E_1 \times E_2$ 
 
 \item When $n>1$ and $p$  a prime of good reduction this conjecture is a consequence of a conjecture of Parshin and Soul\'{e} which asserts that the higher Chow groups of a smooth projective variety over a finite field is torsion. In particular, in the cases when their conjecture is known, this conjecture follows. 
 
 \item One can also formulate a function field variant of this conjecture - in fact this is used to formulate the Beilinson conjectures in that setup. A special case is discussed in \cite{sree3}. 
 
 \item When $n=1$ it is sometimes  the case that all the elements of  $CH^{m-1}(\XX_p)$ are restrictions of elements of $CH^{m-1}(\XX)$ - for example when $X=\CP^k$. In these cases the conjecture is easily shown to the true as one can construct `decomposable' elements in $CH^m(X,1)$ coming from the product 
 $$CH^{m-1}(X) \otimes CH^1(X,1) \rightarrow  CH^m(X,1)$$
 which can be used to prove surjectivity.  
 
 \end{itemize}
 
 The `Archimedean' Hodge-$\D$-conjecture of Beilinson asserts that the regulator map to Deligne cohomology is surjective \cite{jann}. 
 $$ CH^{m}(X,n) \otimes \R \stackrel{r_{\D}}{\longrightarrow} H^{2m-n}_{\D}(X,\R(m)).$$
This is false in general but was proved for  $K3$ surfaces and Abelian surfaces by Chen and Lewis \cite{chle}. It is still expected to be true if $X$ is defined over a number field. In fact Asakura and Saito \cite{assa} show that for certain generic surfaces a non-Archimedean version of this is false over a $p$-adic field as well. In this case too, however,  one  expects the conjecture to be true for varieties defined over global fields. This is why in the statement of the conjecture one has to assume $X$ is defined over a global field. 

Since the boundary map $\partial$  is a non-Archimedean version of the Beilinson regulator map sometimes we refer to it as the non-Archimedean regulator map.  There are other reasons as well - the dimension of the target space is expected to be the same as the order of the pole  of the local $L$-factor at $p$ at a particular point (\cite{cons}) inasmuch as the target of the Beilinson regulator map, namely the Real Deligne cohomology, has dimension equal to the order of the pole of the Archimedian local $L$-factor at a particular point. 
Our conjecture  -- which amounts to surjectivity of the boundary map -- is sometimes referred to as the {\em non-Archimedean Hodge-$\D$-conjecture}.

We prove the following theorem -- 

\begin{thm} Let $A$ be simple, principally polarized, Abelian surface over a p-adic local field  $\K_p$. Let $\A$ be the N\'{e}ron model of $A$ over  $\OO_{{\K_p}}$ where $p \neq 2$ is an odd prime of good non-supersingular reduction. Let $\A_p$ be the special fibre. Let 
$$\Sigma_{A}=\Ker \{CH^2(\A) \longrightarrow CH^2(A)\}$$
Then,  
\begin{itemize}

\item $\Sigma_A$ is a finite $p$-group. 

\item The boundary map
$$CH^{2}(A,1) \otimes \Q \stackrel{\partial}{\longrightarrow} CH^1(\A_p) \otimes \Q$$
is surjective.
\end{itemize}
\end{thm}
Our theorem is hence a non-Archimedean version of the theorem of Chen and Lewis \cite{chle}.

 The outline of the paper is as follows -- from Spiess \cite{spie}, Section 4, it suffices to prove surjectivity of the boundary map -- as that implies finiteness of $\Sigma_A$. He shows that it is a $p$-group as a consequence of the finiteness. In order to prove surjectivity, we first need to get an understanding of the target space $CH^1(\A_p) \otimes \Q$ or equivalently, the N\'eron-Severi group of the special fibre $\A_p$.  It is often the case that the rank of the N\'{e}ron-Severi group of the special fibre $\A_p$ is greater than than the rank of the N\'{e}ron-Severi group  of $A$. Hence there could be  new cycles in the special fibre which do not come from the restriction of cycles in the total space.  We first describe these cycles geometrically. 
 
In order to do this we have to go to the associated Kummer surface and Kummer plane of the special fibre. This is where we use the assumption on odd characteristic. We then use a theorem of Birkenhake and Wilhelm which shows that a  new cycle corresponds to a particular rational curve on the Kummer plane. 
 
This rational curve pulls back to the union  of two rational curves on the associated Kummer $K3$ surface.  We then use a slight generalization of the work of Bogomolov-Hassett-Tschinkel \cite{BHT} [Theorem 18], to deform this sum of rational curves - which lie in the special fibre of the associated Kummer $K3$ surface of $A$ -  to an irreducible rational curve in  the generic fibre. This is where we have to assume that the reduction is non-supersingular. 

Finally, we use this deformed curve to construct an indecomposable higher Chow cycle in the $K3$ surface which can be transferred to the original Abelian surface.  We then show that this higher Chow cycle bounds the new cycle. 

Constructing `interesting' higher Chow cycles on algebraic varieties is often difficult and finding them often has several implications. In the last section we state some immediate applications of these cycles to torsion in codimension $2$. We also discuss some related questions.

{\em Acknowledgements:}. I would like to thank Spencer Bloch and Chad Schoen for some initial discussions which led to this problem.  Brendan Hassett for some very useful discussions on his results. Najmuddin Fakhruddin for his invaluable help, suggestions  and the proof of Theorem \ref{BHTthm}.  Patrick Brosnan, Owen Patashnick, G.V. Ravindra and Jishnu Biswas for some useful conversations. I would also like to thank their referees for their comments and suggestions.  I would like to thank the Indian Statistical Institute, Bangalore, CRM Montreal,  Universit\'e  Paris-Sud (Orsay) and Institut Henri Poincar\'e for their support and hospitality when this work was done. Finally, during the course of writing this paper two people who influenced me quite a bit - Prof. V. Radhakrishnan, whose remarkable approach to  life  changed mine, and Prof. D. Lal, who encouraged me to `prove more lemmas'  - have passed on. I would like to dedicate this work to their memory.

\section{Notation}
\begin{itemize}
\item $A$ -- a principally polarized Abelian surface over a $p$-adic local field $\K_p$. 
\item $\A$ -- a model of $A$ over the ring of integers $\OO_{\K_p}$.
\item $\A_p$ -- the special fibre of $\A$ over the residue field $k$ - usually assumed to be an ordinary Abelian surface.
\item $\LL_0$ -- the line bundle representing the  principal polarization. 
\item $K_A$ -- the Kummer surface associated to an Abelian surface $A$. 
\item $\TK_A$ -- the associated $K3$ surface. 
\item $K\CP_A$ - the associated Kummer Plane.
\item $\phi$ -- the map from $A \longrightarrow K_A$ induced by $\LL_0^2$.
\item $\CC_0$ -- the line bundle corresponding to $\phi(D)$ for $D \in |\LL_0|$
\item $\FF_0$ -- the hyperplane section on $K_A$. $\FF_0=\CC_0^2$. 
\item $\pi$ -- the map from $K_A \longrightarrow K\CP_A$. 
\item $\nu$ -- the blow-up which gives the minimal resolution of singularities  $\TK_A \longrightarrow K_A$
\item $\LD$ -- a line bundle of invariant $\Delta$. 
\item $\QD$ -- the cycle on the Kummer plane corresponding to the line bundle of $\LD$. 
\item $\QD_1$ and $\QD_2$ -- the two components of $\pi^{-1}(\QD)$ in $K_A$. 
\item $\TQ_i$ -- the strict transform of $\QD_i$. 
\item $\TQQ^{\Delta}$ -- the deformation of $\TQ_1+\TQ_2$. 
\item $\TQ_{\eta}$ -- the generic fibre of $\TQQ^{\Delta}$.  
\item $D^{\Delta}_i$ -- the curve $\phi^{-1}(\QD_i)$ in $A$ representing an extra cycle in the N\'{e}ron-Severi of $\A_p$. 

\end{itemize}

\newpage

\section {Abelian surfaces}

\subsection{The Hodge-$\D$-conjecture for Abelian surfaces}

Let $A$ be an   Abelian surface over a $p$-adic field  $\K_p$ with finite residue field.  In this paper we will always assume that $A$ is principally polarized by a line bundle $\LL_0$ - and write $(A,\LL_0)$ when we wish to stress that fact. Let $\A$ be a model over the ring of integers $\OO_{{\K_p}}$ with special fibre $\A_p$. We assume $A$ has good ordinary reduction at $p$, so the special fibre $\A_p$ is smooth. 

One has a map
$$\dots \longrightarrow CH^2(A,1) \otimes \Q \stackrel{\partial}{\longrightarrow} CH^1(\A_p) \otimes \Q \longrightarrow \dots $$
coming from the localization sequence for higher Chow groups. The conjecture above asserts that the map $\partial$ is surjective. 

In order to prove this conjecture one has to first understand the right hand side - namely the Chow group of the special fibre, and then
construct the higher Chow cycles that bound the cycles in the special fibre. In the next section we describe the Chow group of the special fibre.

\subsection{The N\'eron-Severi group of an Abelian surface}

We want to understand the group $CH^1(\A_p)\otimes \Q$. As $CH^1_{hom}(\A_p)\otimes \Q=0$, this is the same as the rational N\'eron-Severi group 
$$CH^1(\A_p)\otimes \Q \simeq NS(\A_p)\otimes \Q.$$
It is well known that the N\'eron-Severi group can be identified as the  part of  the endomorphism algebra, $\End_\Q(\A_p)$, fixed by the Rosati involution $\dagger$ 
$$NS(\A_p)\otimes \Q \simeq \End_{\Q}(\A_p)^{\dagger}.$$
From Tate's theorem  on the description of the endomorphism algebra \cite{tate} one knows that the algebra of a simple Abelian surface  contains a CM field of degree $4$. In particular, the endomorphism algebra  contains a real quadratic field. If the Abelian surface $\A_p$ is not  simple, it contains  the degenerate quadratic `field' $\Q \oplus \Q$. On the real quadratic field as well as $\Q \oplus \Q$, the Rosati involution acts trivially so the rank of the N\'eron-Severi group is always at least two. 

We would like to get an explicit understanding of the generators of this N\'eron-Severi group.  The principal polarization  $\LL_0$ of $A$ is represented by a genus $2$ curve $C$. In fact $A=J(C)$, the Jacobian of $C$. The closure of this  curve in $\A$  restricted to the special fibre $\A_p$ gives one of the generators of $NS(\A_p)\otimes \Q$. However, it is often the case that there are cycles in the special fiber which are not the restrictions of the closure of  cycles in the generic fibre. To get an understanding of these cycles  one has to do a little more work. For this, we have to look at Kummer surface associated to $A$.

\subsection{The Kummer surface and the Kummer plane}

All the statements  in this section are classical and can be found in \cite{biwi}, for example. 

\subsubsection{The Kummer surface}

Let $A$ be an Abelian surface. The {\em Kummer surface} of $A$ is defined to be the hypersurface in $\CP^3$ 
$$K_A=\phi_{\LL_0^2}(A)$$
where $\phi=\phi_{\LL_0^2}$ is the  map 
$$\phi_{\LL_0^2}:A \longrightarrow \CP^3$$
induced by the square of the principal polarization.  Equivalently,  this can be identified with $A/\{\pm 1\}$ -- so the map $A \stackrel{\phi}{\longrightarrow}
K_A$  is a double cover ramified at the sixteen $2$-torsion points of $A$. It is well known, see \cite{biwi}, for example, that the blow up of
$K_A$ at these 16 points is a $K3$ surface $\nu:\TK_A \longrightarrow K_A$.

\subsubsection{ The Kummer Plane}

Let $\pi:\CP^3 \backslash \{0\} \longrightarrow \CP^2$ be the projection with centre $0$, where $0=\phi(0)$. The map $\pi$ restricted to $K_A$
is a double cover of $\CP^2$ ramified at six lines $L_1,\dots, L_6$. These six lines are tangent to a conic. The six lines meet at 15 points
$\{q_{ij}\}$ where $q_{ij}=(L_i \cap L_j)$.  These points  are the images of the non-zero $2$-torsion points under the map $\pi \circ \phi$.
The collection $K\CP_A=(\CP^2, L_1,\dots, L_6)$ is called the {\em associated Kummer plane} of the Abelian surface $A$.


\[
\fbox{ \xy <.6cm,.3cm>: (-6,-2)*{}="A1" ; (2,5)*{}="A2"
**@{-} \POS?(.5)*^++!D{L_1}, (-5,1)*+{}="B1"; (4,-3)*+{}="B2" **@{-}
\POS?(.5)*_+!UR{L_2}; ?!{"A1";"A2"}
*{\bullet}, (-2,5)*+{}="C1";(3,-3)*+{}="C2" **@{-}\POS?(.5)*^+++!DL{L_3}; ?!{"A1";"A2"} *{\bullet};?!{"B1";"B2"}
*{\bullet},
(.5,5)*+{}="D1";(1,-3)*+{}="D2" **@{-} \POS?(.5)*^++!UR{L_4};?!{"A1";"A2"}
*{\bullet};?!{"B1";"B2"} *{\bullet};?!{"C1";"C2"} *{\bullet},
(-5,0)*+{}="E1";(5,0)*+{}="E2" **@{-}\POS?(.5)*^++!U{L_5}; ?!{"A1";"A2"}
*{\bullet};?!{"B1";"B2"} *{\bullet};?!{"C1";"C2"} *{\bullet};?!{"D1";"D2"}
*{\bullet},
(-6,-1)*+{}="F1";(5,1)*+{}="F2" **@{-}\POS?(.45)*^++!D{L_6}; ?!{"A1";"A2"}
*{\bullet}; ?!{"B1";"B2"}*{\bullet}; ?!{"C1";"C2"}
*{\bullet},?!{"D1";"D2"} *{\bullet},?!{"E1";"E2"}
*{\bullet}
\endxy}
\]
\[
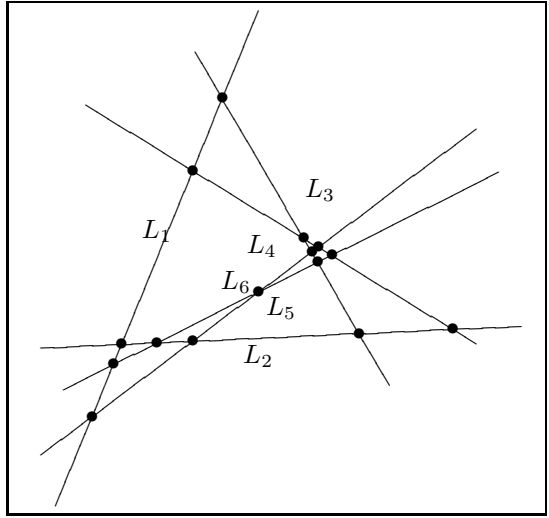
\text{ The six lines and fifteen points on $\CP^2$}
\]

The situation is summarized in the diagram below --
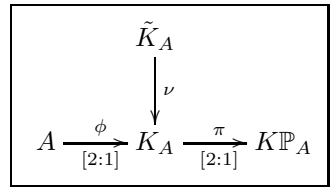
\[
\fbox{

\xymatrix{&\TK_A \ar@{->}[d]^{\nu} \\ A \ar[r]^{\phi}_{[2:1]}& K_{A} \ar[r]^{\pi}_{[2:1]}& K\CP_A } }
\]

\subsubsection{ Humbert's Theorem and its generalizations}
\label{humbertsection}

A classical theorem of Humbert \cite{biwi} states that an Abelian surface $A=J(C)$ has real multiplication by $\ZZ(\frac{1+{\sqrt 5}}{2})$ if and only if there is a conic $Q'$ on the Kummer plane $K\CP_{A}$ which passes through 5 of the  15 points  $\{q_{ij}\}$  {\em and} is tangent to one of the other lines.  Further, since $\End(A)\otimes \Q^{\dagger} \simeq NS(A)\otimes \Q$, the N\'{e}ron-Severi group is of rank at least $2$ and there is a curve $D$ on $A$ such that
$$Q'=\pi \circ \phi(D) $$
and $D$ and $C$ generate the part of the rational N\'eron-Severi group coming from $\ZZ(\frac{1+{\sqrt 5}}{2})$. Hence Humbert's theorem can be viewed as providing a geometric characterization of the extra cycle in the N\'{e}ron-Severi group. 

Birkenhake and Wilhelm \cite{biwi} generalize this theorem  and provide a geometric characterization of the cycles in the N\'eron-Severi group of all Abelian surfaces.  In order to describe their theorem we need some definitions.

The {\em Humbert invariant $\Delta(\LL)$ } of a line bundle  $\LL$ is defined to be
$$\Delta(\LL)=(\LL.\LL_0)^2-2\LL^2
\label{humbert}$$
where $(\;.\;)$ is the intersection pairing on $\Pic(A)$.  This is the negative of the self-intersection number on the orthogonal complement of $\LL_0$ hence is a positive definite quadratic form. There is a line bundle of non-zero Humbert invariant if and only if the Picard number is $>1$.  In fact, having a line bundle on $A$  of Humbert invariant $\Delta$  is equivalent to saying that $\OO_{\Delta} \subset \End(A)$, where $\OO_{\Delta}$ is an order of discriminant $\Delta$. 

To state the theorem of Birkenhake and Wilhelm, we need  to distinguish several different cases of $\Delta$ --
\begin{equation}
\Delta =\begin{cases}I. & 8d^2+ 9-2k\\ II. & 8d (d+1)+ 9 -2k \\III. & 8d^2+8-2k\\IV. & 8d(d+1)+12-2k\\V. & d^2 \;\; ( d>1) \end{cases} \label{deltaBW}
\end{equation}
where $d\geq 1$ and $k \in\{4,6,8,10,12\}$.  The theorem of Birkenhake and Wilhelm is the following --

\begin{thm}[Birkenhake-Wilhelm] Let  $(A,\LL_0)$ be a principally polarized Abelian surface with a line bundle $\LL_{\Delta}$ of invariant $\Delta$.
Then there exists a {\bf rational curve} $\QD$  on the Kummer plane $K\CP_A$ which passes through some of the points $q_{ij}$  with no singularities at those points and meets some of the lines $L_j$ with even multiplicity. One has the following cases --  

   \begin{table}[htbp]
      \centering
      \begin{tabular}{ lcll } 
                       $Case$ &  $\Delta$    & $\deg(\QD)$  & No of points $q_{ij}$\\
       \\
         $I$. & $ 8d^2+ 9-2k$ & $2d$ & $k-1$ \\
         $II$. & $  8d (d+1)+ 9 -2k$    & $2d+1$     &  $k$ \\
         $III$. & $  8d^2+8-2k$   & $2d$   & $k$ \\
         $IV$. & $8d(d+1)+12-2k$      & $2d+1$  & $k-1$ \\
         $V$. &  $d^2$  &  $d-1$ & $3$  \\
        
      \end{tabular}

   \end{table}
Further, $\QD=\pi \circ \phi(D^{\Delta})$  where $D^{\Delta}$ is a  curve on $A$ which lies in the linear system of divisors of  a  line bundle
$\LL$  of the form $\LL_0^a \otimes \LL_{\Delta}^b$ with $b \neq 0$. In particular, the class of $D^{\Delta}$ is not a multiple of the class of the principal polarization. 

\label{biwithm}
\end{thm}

\begin{rem} The curves $C$ and $D^{\Delta}$ generate a two dimensional subspace of the  rational  N\'eron-Severi group.  While $D^{\Delta}$ need not be of Humbert invariant $\Delta$ the divisor  corresponding to the line bundle $\LL_{\Delta}$ and $D^{\Delta}$  span the same subspace.  An important aspect of their theorem is  that there is a {\em rational curve} on $K\CP_A$ which represents the extra cycle. Humbert's theorem above is the special case when $\Delta=5$. 
\label{span}
\end{rem}

We now lift this cycle to the Kummer surface. We have the following lemma, which is proved by Jakob \cite{jako} in a special case --

\begin{lem} $\pi^{-1}(\QD)=\QD_1 \cup \QD_2$, where $\QD_1$ and $\QD_2$ are  rational curves on $K_A$.
\end{lem}

\begin{proof} Birkenhake-Wilhelm \cite{biwi} show that there is a curve $\QD_1$ such that the map $\pi:\QD_1 \longrightarrow \QD$ is birational.
Hence the curve $\QD_1$ is rational. Since the map $\pi$ is a double cover, $\pi^{-1}(\QD)=\QD_1 \cup \QD_2$ where $\QD_2$ is another rational
curve.
\end{proof}

\subsubsection{The involution $\iota$.} 
\label{involutionspan} $\pi$ is a double cover so it induces an involution $\iota$  on $K_A$. Under this involution one has
$$\iota(\QD_1)=\QD_2.$$
Since $\pi$ is ramified over the lines $L_i$, $\pi^{-1}(L_i)$ is fixed under $\iota$ and we will abuse notation to denote the line $\pi^{-1}(L_i)$ in $K_A$ by $L_i$ as well. This involution also acts on the N\'{e}ron-Severi group of the Abelian surface. We know that $\Q(\sqrt{\Delta})$ can be identified with a subgroup of $NS(A)\otimes \Q$ and with respect to this  identification, $\iota$ can be though of as the non-trivial Galois conjugation.

The rational N\'eron-Severi group $NS (\CP^2)\otimes \Q \simeq \Q$. So the pull-back of a cycle in $K\CP_A$ lies in  subspace  $\Q\FF_0$. The involution $\iota$ acts trivially on the classes of such cycles and the class of a  cycle in $K_A$ lies in the $\Q$-span of $\FF_0$ if and only if its class  is fixed by $\iota$. 

The work of Birkenhake and Wilhelm is in the complex situation but their work is purely algebraic and carries through, {\em mutatis mutandis}, to the case of Abelian surfaces over finite fields as long as the characteristic is not $2$.  

We apply this to the case when $A=\A_p$, an Abelian surface over a finite field. Hence, for a line bundle  $\LD$  in $\A_p$ with Humbert invariant $\Delta >0$ there is a rational curve $\QD$ in $K\CP_{\A_p}$ associated to this line bundle. This cycle pulls back to $\QD_1 \cup \QD_2$ on $K_{\A_p}$ and one of the components pulls back to a cycle $D_p^{\Delta}$ in $\A_p$ whose cycle class is not a multiple of the class of the principal polarization. 

\section{Rational curves on $K3$ surfaces}

Our next step is to deform the rational curves $\QD_1 \cup \QD_2$  above to the generic fibre. For this, we use some recent work of Bogomolov, Hassett and Tschinkel  on the existence of rational curves on $K3$ surfaces. A conjecture, attributed to Mumford, states that there are infinitely many rational curves on an arbitrary $K3$ surface. The existence of even a single rational curve on a  $K3$ surfaces is not always trivial. 
 
In a recent paper, Bogomolov, Hassett and Tschinkel \cite{BHT} proved a mixed characteristic generalization of a result of Mori and Mukai
\cite{momu} which they use to construct infinitely many rational curves on certain $K3$ surfaces of Picard number 1. Their  idea is to deform rational curves from a special fibre to the generic fibre.  We would like to do something similar and  for our purposes we need to use  a slight modification Theorem 18 of \cite{BHT}. The proof of this was explained to us by N. Fakhruddin\cite{fakh}.

\begin{thm} 

  Let $(S_0,f_0)$ be a $K3$ surface over a finite field $k$ of
  characteristic $p$ together with a divisor class $f_0$ on
  $S_0$. Suppose
$$C=C_1+\dots+C_r$$
is a connected union of rational curves $C_i \subset S_0$ such that
$[C]=Nf_0$ for some positive integer $N$. Assume that the $C_i$ are
distinct and that $S_0$ is not uniruled.  Let $(S,f)$ be a
(projective) $K3$ surface together with a divisor class defined over a
finite extension $W'$ of the ring of Witt vectors $W(k)$ reducing to
the base change of $(S_0,f_0)$ to the residue field of $W'$.  Then
there is a relative curve $R \subset S\times_{Spec(W')}Spec(W'')$,
where $W''$ is a finite extension of $W'$ such that $R$ reduces to
(the base change of) $C$ and all irreducible components of the generic
fibre of $R$ are rational. If there are only two curves $C_1$ and $C_2$, the rational curve $R$ is irreducible.

 \label{BHTthm}
\end{thm}

\begin{proof}
  The proof is essentially the same as that of \cite{BHT}, Theorem 18 -
  there the authors assume that $f_0$ is an ample class so here we
  only explain, using their notation,  why that assumption is
  unnecessary.

  As in \cite{BHT} the dimensions of the formal scheme
  $\overline{\mathcal{M}}_0^{\circ}(\mathcal{S}/B, Nf_0)$ and its
  image in $B$ is at least $20$; the ampleness assumption is not used
  anywhere for this computation.  Moreover, Theorem 16 of
  \cite{BHT} does not have the ampleness hypothesis, so as there we also
  get that the formal scheme $\Sigma_{Nf_0}$ has dimension $20$ and is
  not contained in the fibre over the closed point of $Spf(W(k))$.
  
  By construction, the image of
  $\overline{\mathcal{M}}_0^{\circ}(\mathcal{S}/B, Nf_0)$ in $B$ is
  contained in $\Sigma_{Nf_0}$, so since they have the same dimension
  and $\Sigma_{Nf_0}$ is smooth (by Proposition 16 of \cite{BHT}), they
  must be equal.

  By assumption, since $f$ reduces to $f_0$ it follows that the
  morphism $Spec(W') \to B$ corresponding to the relative surface $S$
  factors through $\Sigma_{Nf_0}$, hence also though the image of
  $\overline{\mathcal{M}}_0^{\circ}(\mathcal{S}/B, Nf_0)$. Since we
  have assumed that $S$ is projective over $W'$, it follows from
  Grothendieck's algebraisation theorem as in \cite{BHT}  that the
  restriction of $\overline{\mathcal{M}}_0^{\circ}(\mathcal{S}/B,
  Nf_0)$ to $Spec(W')$ is an \emph{algebraic} stack. Moreover, this
  stack maps onto $Spec(W')$ so the generic fibre is
  non-empty. It follows that there exists a morphism
$$ Spec(W'') \to \overline{\mathcal{M}}_0^{\circ}(\mathcal{S}/B, Nf_0)
\times_B W' \ , $$
with $W''$ a finite extension of $W'$ such that the image of the
closed point of $Spec(W'')$ corresponds to the curve $C$. Pulling back
the universal family over
$\overline{\mathcal{M}}_0^{\circ}(\mathcal{S}/B, Nf_0)$ gives us a
stable map to $S$ (defined over $W''$) and the image of this gives us
the desired relative curve $R$.

In the special case when there are only two curves $C_1$ and $C_2$
which do not deform to the generic fibre Lemma 19 of \cite{BHT} shows that $R$ is irreducible. 
\end{proof}


\section{ Elements of the Higher Chow Group}

Let $X$ be a surface over a global or local field  $\K$. The group  $CH^2(X,1)$ has the following presentation \cite{rama}. It is generated by
formal sums of the type
$$\sum_i (C_i,f_i)$$
where $C_i$ are curves on $X$ and $f_i$ are $\bar{\K}$-valued functions on the $C_i$ satisfying the co-cycle condition
$$\sum_i \div{f_i}=0.$$
Relations in this group are give by the tame symbol of pairs of functions on $X$.

There are some  elements of this group coming from the product structure
$$\bigoplus_{\BL/\K} CH^1(X_{\BL}) \otimes CH^1(X_{\BL},1) \longrightarrow \bigoplus_{\BL/\K}  CH^2(X_{\BL},1) \stackrel{ \oplus N_{\BL/\K}}{\longrightarrow} CH^2(X,1)$$
where $\BL$ runs through all finite extensions of $\K$ and $N_{\BL/\K}$ is the norm map.  The subgroup  of {\em decomposable elements} $CH^2_{dec}(X,1)$ is  the image of this map.  A theorem of Bloch \cite{bloc} [Theorem 6.1], says that $CH^1(X_{\BL},1)$ is simply $\BL^{*}$ where $\BL$ is the field of definition of $X_{\BL}$ -- so such an element looks like a sum of elements of the type $ (C, a ) $,  where $C$ is a curve on $X_{\BL}$ and $a$ is in $\BL^*$. The group of {\em indecomposable} elements of $CH^2(X,1)$ is the quotient group 
$$CH^2_{ind}(X,1) \simeq CH^2(X,1)/(CH^2_{dec}(X,1).$$ 
In general it is not so easy to show that this  group is non-trivial  and in some instances, for example for $X=\CP^2$, it is trivial. 

\begin{rem}The group $CH^2(X,1)\otimes \Q$ is the same as the ${\mathcal K}$-cohomology group $H^1_{Zar}(X,{\mathcal K}_{2})\otimes \Q$ and the motivic cohomology group
$H^3_{\M}(X,\Q(2))$ - see, for example, \cite{stac}.
\end{rem}

\subsection{The boundary map}

Let $X$ be as above and $\XX$ a  model of $X$  over the ring of integers with special fibre  $\XX_p$ at a prime $p$. We assume $\XX_p$ is smooth - that is, $X$ has good reduction at $p$. The boundary map 
$$\partial:CH^2(X,1) \longrightarrow CH^1(\XX_p)$$
is defined as follows 
$$\partial \left( \sum_i (C_i,f_i) \right)=\sum_i \div_{\bar{C_i}}(\bar{f_i})$$
where $\bar{C_i}$ denotes the closure of $C_i$ in the semi-stable model $\XX$ of $X$. From the co-cycle condition, the horizontal divisor,
namely, the closure  $\sum_i \overline{\div_{C_i}(f_i)}$ of $\sum_i \div_{C_i}(f_i)$ ,  is $0$, so the boundary is supported on the special fibre.

For a decomposable element of the form $(D,a)$ the boundary map is particularly simple to compute 
$$\partial((D,a))=\ord_p(a) {\mathcal D}_p.$$
where ${\mathcal D}_p$ is the special fibre of the closure of $D$ in $\XX$. 
\begin{rem}In particular, a cycle in the special fibre which is not the restriction of a cycle in the generic fibre cannot appear in the  boundary of a decomposable element. \label{indres} \end{rem}

\section{A new element in the higher Chow group of an Abelian Surface}

Let $(A,\LL_0)$ be an Abelian surface over a p-adic local field $\K_p$  as before.  We assume  that $A$ has good, non-supersingular reduction at $p$. Then we have 
$$\pic (\A_p)\geq \pic(A)$$
where $\pic$ is the Picard number - the rank of the N\'{e}ron-Severi group.  If $\pic (\A_p)=\pic (A)$, then every cycle in $NS(\A_p)$ is the restriction of the closure of a cycle in $NS(A)$,  so the surjectivity of the boundary map is trivial -- as all the cycles can be obtained as boundaries of decomposable elements. Hence we assume $\pic(\A_p) >\pic(A)$.  In other words, we assume that there are always  cycles in the N\'{e}ron-Severi group of the special fibre which are not the restriction to the special fibre of the closure  of cycles on the generic fibre. 

In this section for such a cycle $\D_p$  we will construct an an element $\Xi_{A,\D_p}$ of $CH^2(A,1)$ such that the image $\partial(\Xi_{A,\D_p})$ 
is a non-zero multiple of the cycle $\D_p$ in the N\'{e}ron-Severi group of $\A_p$. In particular, from Remark \ref{indres}, it is an {\em indecomposable} element in the higher Chow group. 

The idea is the following:  One way to construct a higher Chow cycle on a surface $S$  is to use a rational curve $C$ with a node $P$ lying on the surface. The normalisation  $\eta:\TC \to C$ will have two points $P_1$ and $P_2$  lying over the node. Since $\TC$ is rational one can find a function $f$ on $\TC$ with divisor $P_1-P_2$. Then the element $\eta_*(\TC,f)$ is an element of $CH^2(S,1)$. 

Unfortunately, Abelian surfaces do not have rational curves. So we work with the associated Kummer $K3$ surface. The work of Birkenhake and Wilhelm \cite{biwi} shows  that given a cycle in the special fibre  of non-zero Humbert invariant  there is a  rational curve in the special fibre of the $K3$ surface which represents it.  We show that this curve along with a conjugate  rational curve satisfies the condition of  Theorem \ref{BHTthm} - so we can deform them to an irreducible rational curve on the generic fibre. Further, we show that this curve has a node. Using that curve we construct a  higher Chow cycle on the generic Kummer surface, transfer it to the Abelian surface and show that it has the required properties. 

Since we work with related objects on the Abelian surface and its associated Kummer surface, $K3$ surface and Kummer plane, there is a lot of notation which can get a little confusing so the reader should keep the above remarks in mind.

\subsection{ Lifting to the model }

Let $\A$ be a regular model of $A$ with special fibre $\A_p$. Let $K_{\A}$ be the Kummer Surface of $\A$. Let $\FF_0$ denote the hyperplane section on $K_{\A}$ -- so $\phi^{*}(\FF_0)=\LL_0^2$. Let $(\TK_{\A},\TFF_0)$ denote the $K3$ surface obtained by blowing up the sixteen nodal points on $(K_{\A}, \FF_0)$ and
$$\nu:\TK_{\A} \longrightarrow K_{\A}$$
denote the birational map from $\TK_{\A}$ to $K_{\A}$. Let $\nu^*(\FF)$ denote the total transform of a line bundle $\FF$ on $K_{\A}$ and $\TFF_0$
denote the strict transform of $\FF_0$. Let $\TCC_0$ denote the strict transform of $\CC_0$ - where $\CC_0$ is  a cycle whose class is the generator of $\Q\FF_0 \cap  NS(K_{\A})$. Its special fibre $\TCC_{0,p}$ is the generator of $\Q\FF_{0,p} \cap NS(\TK_{\A_p})$ since the co-kernel of the specialisation map has no torsion.

Since the  cycle $\D_p$  in the N\'{e}ron-Severi of $\A_p$  is not in the span of the polarisation, it has a non-zero Humbert invariant $\Delta$. Hence from  Theorem \ref{biwithm} there is a rational curve $\QD$ on $K\CP_{\A_p}$ corresponding to this cycle. Let $\QD_1, \QD_2$ and $\TQ_1,\TQ_2$ the rational curves  lying above this curve in $K_{\A_p}$ and $\TK_{\A_p}$ respectively. $\TQ_i$ is the strict transform of $\QD_i$ under the birational map $\nu$. Let $D^{\Delta}_1$ and $D^{\Delta}_2$ denote the pullbacks of $\QD_1$ and $\QD_2$ to $\A_p$.

\begin{lem} The cycle $\TQ_1+\TQ_2 \in |\TCC_{0,p}^{N}|$ for some $N \in \ZZ$.
\end{lem}

\begin{proof} Since $\Pic(\CP^2) \simeq \ZZ$,  the class of $\QD$ is a multiple of the hyperplane section. Its pull back  $\pi^*(\QD)$  is fixed by the involution $\iota$ and so its class  lies in the part of $K_{\A_p}$ spanned by the hyperplane section -- which is $\ZZ [\CC_{0,p}]$ by the remarks in Section \ref{involutionspan}. Hence the strict transform $\widetilde{\pi^*(\QD)}$ lies in $\ZZ [\TCC_{0,p}]$. However, $\pi^{-1}(\QD))=\QD_1 \cup \QD_2$. Hence its  strict transform is $\TQ_1 \cup \TQ_2$ and one has
$$\TQ_1+\TQ_2 \in | \TCC_{0,p}^N|$$
for some $N \in  \ZZ$.

\end{proof}

We would like to apply Theorem \ref{BHTthm}  to this situation. For that we need the curve $\TQ_1+\TQ_2$ to be connected. 
To verify that this is the case, we analyze the points of intersection of $\QD_1$ and $\QD_2$ under the blow-up . The curves $\QD_1$ and $\QD_2$ intersect at several points on $K_{\A_p}$ -- in fact, one has 

\begin{lem} If $x \in L_i \cap \QD_1$ for some $i$ then $x \in  \QD_1 \cap \QD_2$ . Further --
\begin{itemize}
\item  $x$ lies in the image of the two torsion if and only if $x \in L_i \cap L_j \cap \QD_1 ( \cap \QD_2)$ for some $i$ and $j$. 
\item  The other points of  $L_i \cap \QD_j$ which are not the images of two torsion points appear with even multiplicity. 
\end{itemize}
\label{intersection}
\end{lem}

\begin{proof}  If $x$ lies on $L_i \cap \QD_1$, then, as $\iota$ fixes $L_i$, $\iota(x)=x$. Hence $x$ lies on $\iota(\QD_1)=\QD_2$.  For the second and third statements we refer to \cite{biwi}, Lemma 6.1.


\end{proof}

From \cite{biwi} Thms 7.1-7.5  we know there are $k$, $k-1$ or $3$ points on $\QD$ which are the images of the  $2$-torsion points of $\A_p$. These points lie on $L_i \cap L_j \cap \QD$, hence their pre-images under $\pi$ lie on $\QD_1 \cap \QD_2$.   Under the blow-up to $\TK_{\A_p}$  they  need not be points of intersection of $\TQ_1 \cap \TQ_2$ any longer. The remaining points on $L_i \cap \QD$ are points of even multiplicity. Hence we need to show that either there is a point of intersection which does not lie in the image of  the two-torsion points or that a point which lies in the image of the  two-torsion points continues to lie on the intersection after the blow-up.  We abuse notation and call a point of intersection of $\QD_1$ and $\QD_2$ on  $K_{\A_p}$, or its image in  $K\CP^2_{\A_p}$,  a {\em two-torsion point} if it lies in the image of the two-torsion points of $\A_p$. Similarly, we call such a point a  {\em non two-torsion point} if it does not lie in the image of the two-torsion points of $\A_p$.  


\begin{lem} 

If $d>2$ then $\TQ_1\cup \TQ_2$ has at least two points of intersection and, in particular, is connected. 
\end{lem}

\begin{proof} From \cite{biwi}, we know that the degree of $\QD$ is either $2d$, in cases $I$ and $III$, or $2d+1$, in cases $II$ and $IV$ of the table  \eqref{deltaBW} in Section \ref{humbertsection}. Further, there are at most five two-torsion points of  on a line $L_i$ - as these correspond to points of intersection with the other $5$ lines $L_j$.   

\begin{itemize}
\item  {\em Case 1 -- The two-torsion points on $\QD \cap L_i$ appear with multiplicity $1$:} Since the non two-torsion points on $L_i$ appear with even multiplicity, there are at least $2d-4$ non two-torsion points as, if the degree is $2d$ there can be at most four two-torsion points and if the degree is $2d+1$ there can be five. 
Hence if $2d-4>1$, there will be at least $1$ such point on every line $L_i$. Such points do not get separated in the strict transform to $\TQ_1 \cup \TQ_2$ hence this curve has at least two points of intersection - in fact at least six. 
\item {\em Case 2 -- The two-torsion points have  multiplicity $>1$:} If at least  two two-torsion points on $(\bigcup_i L_i) \cap \QD$ have high multiplicity, then, under the blow up to $\TQ_1 \cup \TQ_2$  they will remain singular and points of intersection of $\TQ_1 \cup \TQ_2$ -- see Case 2 of Lemma \ref{sing} for details. If there is only one, then Case $1$ above  shows there is at least one non two-torsion point on $\TQ_1 \cup \TQ_2$. 
\end{itemize}

Since there are at least two points of intersection, the curve $\TQ_1 \cup \TQ_2$ is connected. 

\end{proof}

\begin{rem} The restriction that $d>2$ is not a very serious one as  --  if $\A_p$ contains a line bundle of invariant $\Delta$ then $\A_p$ contains a line bundle of invariant $m^2\Delta$ for any $m \in \ZZ$. For $m^2\Delta$ the corresponding $d$ will be larger. Further, even if $d=1$ or $d=2$ for most values of $k$ there is no problem. 
\end{rem}

Hence  $\TQ_1 \cup \TQ_2$ along with $(\TK_{\A},\TCC_0^N)$ satisfy the hypothesis of Theorem \ref{BHTthm}. Hence  one has that  there exists an irreducible rational curve $\TQQ^{\Delta}$ with $\TQQ^{\Delta} \in |\TCC_0^N|$ defined  on some finite extension of the Witt vectors of $W(k)$ which splits into a sum of these two rational curves mod $p$ --
$$\TQQ^{\Delta}_p=\TQ_1\cup\TQ_2.$$

\subsection{Deformation of Singularities}

A {\em node} is a singularity which is locally isomorphic to the singularity at the origin of the plane curve  $y^2=x^{2k}$.  The number $2k$ is called the {\em order} of the node. A node of order $2$ is called an {\em ordinary node}.   Another way of describing it is as a point where two smooth branches of a curve meet. We will used the word `node' to refer to a node of any order and will use the word `ordinary node' for a node of order $2$. A nodal singularity can be can be resolved by a sequence of blow-ups - each reducing the order by $2$ and hence at the penultimate stage one has an ordinary node and in the normalisation there will be two points lying over the node. 

We would like to use the generic fibre $\TQ_{\eta}$ of the curve $\TQQ^{\Delta}$ to construct a higher Chow cycle as described in the beginning Section 6.  Here we will show that the curve  $\TQ_{\eta}$ has a node.  For this we have to understand the singularities coming from the intersection of $\TQ_1$ and $\TQ_2$ and their deformations. 

\begin{lem}  Let $P$ be a singularity of the curve $\TQ_1 \cup \TQ_2$.  Assume $\nu(P)$ lies on $L_i\cap \QD_1 \cap \QD_2$ for some $i$. Then the singularity at  $P$ is \'{e}tale locally isomorphic to a  singularity  at the origin of a plane curve of the type $y^2=x^{2k}$ for some positive integer $k$ -- that is, it is a higher order node. 
\label{sing}
\end{lem}
\begin{proof}  From Lemma \ref{intersection}, the singularity $\nu(P)$ is either the image of a  two-torsion point  or is a  double intersection of a line $L_i$ with $\QD_j$. Let $P'=\pi(\nu(P))$, the image of $P$ in $K\CP^2_{\A_p}$ which lies on the curve $\QD$.  Since $\QD$ is a rational curve,  after going to a completion we may assume that it is locally isomorphic to  $y=F(x)$ near $P'$, where $F(x)$ is a power series. 
\begin{itemize}

\item{\em Case 1:} If $P'$ is not the image of a two torsion point, then from  Lemma \ref{intersection}, the multiplicity of $L_i \cap \QD$ is {\em even}.  Hence near $P'$, $\QD$ is locally isomorphic to $y=x^{r}G(x)$ with $r=2k$ even and $G(0) \neq 0$ and  $L_i$ is locally isomorphic to  $y=0$. Replacing $y$ by $y/G(x)$ we may assume that it is locally isomorphic to $$y=x^{r}.$$

The map $K_{\A_p} \stackrel{\pi}{\longrightarrow} K\CP^2_{\A_p}$ in a neighborhood of $\nu(P)$ is a double cover ramified at the line $L_i$ so it is locally isomorphic to $z^2=y$ near  $\nu(P)$, with the double cover given by $(x,y,z) \longrightarrow (x,y)$.

The pre-image of the curve $\QD$ has two components $\QD_1$ and $\QD_2$ which are birationally isomorphic to $\QD$ and  are conjugate via the map $(x,y,z) \rightarrow (x,y,-z)$. So locally, near the ramified point $\nu(P)$ on $K_{\A_p}$, the curve is given by the equations 
\begin{center}
\begin{tabular}{lcr}
$z^2=y$ &and &  $y=x^{2k}$
\end{tabular}
\end{center}
so $z^2=x^{2k}$. Since  $\nu$ is a local isomorphism near $P$ the curve is locally isomorphic to $z^2=x^{2k}$ near $P$ as well. Hence it is a node of order $2k$.

\item{\em Case 2:} If $P'$ is the image of a two torsion point, then, as before, on $K\CP^2_{\A_p}$  the curve is locally isomorphic to 
$$y=x^{r}$$
for some natural number $r$ near $P'$. Here the two lines $L_1$ and $L_2$ are given by the axial lines $x=0$ and $y=0$. The double cover is ramified at both the lines $x=0$ and $y=0$ so is locally isomorphic to the surface $z^2=xy$ with the map given by  $(x,y,z) \longrightarrow (x,y)$.

However, unlike the earlier case, the point $\nu(P)$ is singular. The map $\nu:\TK_{\A_p} \longrightarrow K_{\A_p}$  is a blow up  in a neighborhood of $\nu(P)$. Hence we have to see what happens to the singularity under this blow up. 

Let $[u,v,w]$ be the co-ordinates on the exceptional fibre $\CP^2$ of the blow up of $\AF^3$ at the origin. The  blow up of $z^2-xy$ at $(0,0,0)$ is given by the equations
\begin{center}
\begin{tabular}{lr}
$xv=yu$ & $xw=zu$ \\
$yw=zv$ & $z^2=xy $
\end{tabular}
\end{center}

The blow up of $\AF^3$ at the origin is covered by three affine open sets $\AF^3 \times \AF^2_u$, $\AF^3 \times \AF^2_v$ and $\AF^3 \times \AF^2_w$, where $\AF^2_*$ is the plane given by $*=1$ in the exceptional fibre $\CP^2$. To understand the strict transform of the curve $y=x^r$ in the blow up of the surface $z^2=xy$,  we first restrict to $\AF^3 \times \AF^2_u$. This gives us the equations 
\begin{center}
\begin{tabular}{lr}
$y=xv$ & $z=xw$ \\
$z^2=xy $ & $y=x^{r}$
\end{tabular}
\end{center}
 which implies 
$$z^2=x^2w^2=xy=x^{r+1}$$
and so 
$$x^2w^2=x^{r+1}$$
so in a neighborhood of the point $(0,0,0,[1,0,0])$ in the blow up of $\AF^3$  the curve is locally isomorphic to 
$$w^2=x^{r-1}$$
Since we know  the curve $\nu^{*}(\pi^{*}(\QD))=\TQ_1 \cup \TQ_2$ has two components this forces $r-1$ to be {\em  even}. On  $\AF^3 \times \AF^2_v$ and $\AF^3 \times \AF^2_w$ there are no singularities. Since we have assumed that $P$ is a singular point we have $r>1$ and it is a nodal singularity of order $r-1$.

\end{itemize}

From Artin, \cite{arti}, Corollary 2.6, the original singularity and this are \'{e}tale locally isomorphic. 

\end{proof}

A local deformation of a singularity is a family $\XX\slash T$ where $T$ is the spectrum of  a complete local ring such that the special fibre $\XX_0$ is a variety with the  singularity. It is said to be {\em miniversal} if further 
\begin{itemize}
\item (versal) For any other deformation $\XX^{\prime} \slash S$, with $S$ the spectrum of a complete local ring, there is a morphism $\Phi: S \to T$ such that $\XX^{\prime}$ and $\XX \times_T S$ become isomorphic after completing along the closed fibre over zero.
\item (mini) Although $\Phi$  may not be unique, the induced map on Zariski tangent spaces of $S$ and $T$ is uniquely determined.
\end{itemize}
The miniversal deformation of a singularity is \'{e}tale local, so for the purposes of computing this we may assume that the singularity at $P$ is of the form $y^2=x^{2k}$. Using  \cite{hart}, [Theorem 14.1] and Elkik \cite{elki}, [Theoreme 8] one can see that the miniversal deformation space of the singularity over $Spec(W(k))$ is isomorphic to 
$$Spec(W(k)[x,y,a_i,\dots,a_r]/(y^2-F(x)))$$
where 
$$F(x)=\prod_{i=1}^{r} (x-a_i)^{n_i}$$
and  $\sum_{i=1}^{r}  n_i a_i=0$ and $\sum_{i=1}^{r} n_i=2k$. So $F(x)$ is a monic polynomial of degree $2k$ with no degree $2k-1$ term.  This  has singularities at $a_i$ if $n_i>1$.  

For a curve $X=\bigcup_{i=1}^{s_X} X^i$, where $X^i$ are the irreducible components,  the arithmetic genus is 
$$p_a(X)=\sum_i p_g(X^i) -(s_X-1)+ \sum_P \delta_P$$
where $p_g(X^i)$ is the geometric genus (genus of the normalization) and
$$ \delta_P=\sum_Q \frac{m_Q(m_Q-1)}{2}$$
where the sum is over all the infinitely near points $Q$  of $P$ including $P$ (that is, points lying over $P$ under a series of blow-ups.) with multiplicity $m_Q$. Note that $\delta_P \neq 0$ if and only if $P$ is a singular point. We call $\delta_P$ the contribution from the singularity $P$.

\begin{lem} Let $\XX/T$ be a flat family of projective, rational curves over the spectrum of a complete dvr. Suppose 
\begin{itemize}

\item The closed fibre $\XX_0$ has two rational components  and has at least two nodes of some order.  
\item The generic fibre $X$ is irreducible. 

\end{itemize}
Then there is a node $R$ in the special fibre such that 
$$\sum_t \delta_{R^t} = \delta_{R}$$
where $R^t$ are the singular points of $\XX$ specializing to $R$. In particular, this implies that the generic fibre too has a nodal singularity.
\label{nodelemma}
\end{lem}

\begin{proof}

For  a singularity of multiplicity $2$ at the origin of a plane curve  of  the type $y^2=x^r$  one  has 
$$\delta_{(0,0)}=\begin{cases} r/2 & \text{ if $r$ is even }\\ (r-1)/2 & \text{ if $r$ is odd} \end{cases}$$
This is because  the blow up at the origin of $y^2=x^r$ has  a singularity of type $y^2=x^{r-2}$, so the order is lowered by  $2$.  The contribution from this blow up -- and every subsequent blow up -- is $2(2-1)/2=1$. One can repeat  this  process till the point is non-singular -- which happens when the penultimate stage is an ordinary node or a cusp. For an ordinary node $y^2=x^2$ and for a cusp  $y^2=x^3$ one has  $\delta_{0,0}=1$ and from this the result follows. 

In our case, $\XX_0=\XX_0^1 \cup \XX_0^2$, so $s_{\XX_0}=2$. Since the irreducible components are rational,  if the  singularities of $\XX_0$ are at points $P$  one has 
$$p_a(\XX_0)=0-1+\sum_{P \in \XX_0}  \delta_{P}$$
$\XX$ is a flat family so the arithmetic genera of the fibres are constant -- that is,  $p_a(\XX_0)=p_a(X)$. From this along with  $p_g(\XX_0)=p_g(X)=0$  and the fact that the generic fibre is irreducible, so $s_X=1$, one has 
\begin{equation}
-1+\sum_{P \in \XX_0} \delta_P=\sum_{R \in X} \delta_{R}
\label{genspec}
\end{equation}
Let $\{P^j\}_{j\in J}$ denote  the points on $X$  whose closure in $\XX$  intersects the special fibre $\XX_0$ at  a singular point  $P$. These points are said to {\em specialize} to $P$.  Then one always has 
$$\sum_j \delta_{P^j} \leq \delta_{P}$$
Now suppose $\sum_j \delta_{P^j} < \delta_{P}$ for some singularity $P$. From \eqref{genspec} one has  
\begin{itemize}
\item $\sum_j \delta_{P^j} =\delta_{P}-1$ 
\item $\sum_t \delta_{R^t} = \delta_{R}$ for all other singularities $R$. 
\end{itemize}

In particular, since we have assumed there are at least $2$ nodes, there is a node $R$ such that  $\sum_t \delta_{R^t} = \delta_{R}$. We claim that all the singularities $R^t$ which specialize to $R$ are nodes. 

Since  $R$ is a  node, it is of the form $y^2=x^{2k(R)}$ with  $\delta_R=k(R)$. Let   $(a_{t},0)$ with $n_{t}>1$ be the singular points on the generic fibre which specialize to  $R$, where locally the curve looks like 
$$y^2=\prod_t (x-a_t)^{n_t} $$
with $\sum_t {n_t}=2k(R)$.  If $n_t$ is odd and greater than $1$, then from the remarks above the singularity $(a_t,0)$ will have $\delta_{(a_t,0)}=\frac{(n_t-1)}{2}$.  Further, as $\sum_t n_t=2k(R)$ is even, there will have to be at least $2$  points with odd exponent. So in that case we have 
$$\sum_t \delta_{(a_t,0)}= k(R)-1.$$ 
This contradicts the assumption that 
$$\sum_t \delta_{R^t}=\delta_R=k(R)$$
Hence for this node, all the points $R^t$ have to be nodes, so we have at least one node on the generic fibre. 
 
\end{proof}

\subsection{ The Higher Chow Cycle}

We can apply the results of the previous section to $\TQ_1 \cup \TQ_2$ and so  we may assume that there is an irreducible curve $\TQQ=\TQQ^{\Delta}$ defined over some extension of the Witt vectors of the residue field such that the special fibre is $\TQ_1 \cup \TQ_2$. Further, there is a node $R$ on the special fibre such that 
$$\sum_t \delta_{R^j}=\delta_R$$
where $R^j$ denote the points on the generic fibre specializing to $R$ and we can assume $R$ lies on the intersection of the two components as $\TQ_i$ are non-singular outside the image of $0$.  There is another singular point $P \neq R$ such that 
$$\sum_t \delta_{P^t}=\delta_P-1$$

\begin{lem} Let $R$ be a node in the special fibre $\TQ_1 \cup \TQ_2$ such that
\begin{equation}
\sum_j \delta_{R^j}=\delta_R
\label{delta}
\end{equation}
Let $R'$ be a point lying over $R$ under a blow up of a section $\bar{R}^j$ of $\TQQ$  specializing to $R$ and $R^j$. Let $R'^t$ denote the points in the generic fibre specializing to $R'$. Then 
$$\sum_t \delta_{R'^t}=\delta_{R'}$$ 
\label{blowup}
\end{lem}

\begin{proof}
From Lemma \ref{nodelemma}, we know that the $R^j$s are nodes. Consider a single blow up of the closure $\bar{R}^1$ of a  point $R^1$ in the model.  Since $R^1$ is a node there is either a node of one  lower order or two non-singular points lying over $R^1$.  Suppose it is a lower order node $R'^1$. Recall, 
$$\delta_{R^1}=\sum_{Q}  \frac{m_Q(m_Q-1)}{2}$$
where $Q$ runs through all the infinitely near points of $R^1$ including $R^1$. Hence in this case, 
$$\delta_{R^1}=\delta_{R'^1}+2(2-1)/2=1+\delta_{R'^1}$$
The same holds in the special fibre -- under the blow up the point $R'$ lying over $R$  is a node of one lower order. Hence the formula 
$$\sum_j \delta_{R'^j}=\delta_{R'}$$
continues to hold as the other points $R^j$ with $j \neq 1$ are not affected by the blow up. 

This formula also  holds when there are two smooth points over $R^1$ as in that case $R^1$ is an ordinary node so $\delta_{R^1}=1$ and $\delta$ of a smooth point is $0$.  

\end{proof}

\begin{cor} If $R^N$ is a point on the generic fibre of the normalization lying over a node $R^1$ which specializes to a node $R$ in the special fibre satisfying equation \eqref{delta} above, then $R^N$ specializes to a non-singular point on the special fibre of the normalization. \label{corblowup}
\end{cor} 
\begin{proof}

Iterating the process above, we see that as long as there is a singular point $R'$ lying over $R$ - which has to be a node -  there are nodal points $R'^j$ specializing to $R'$. In the normalization, there are no singular points in the generic fibre, hence a point on the generic fibre specializing to a point  lying over $R$ can no longer be singular. 

\end{proof}

Let $\Psi:{\mathcal N}^{\Delta}  \rightarrow \TQQ^{\Delta}$ denote the normalization of $\TQQ$ and $\Psi:N_{\eta} \rightarrow \TQ_{\eta}$ the generic fibre of ${\mathcal N}^{\Delta}$. Let $R^{\TQ}$ be a node of $\TQ_{\eta}$ specializing to $R$. Under the normalization, there are two non-singular points $R^{N}_1$ and $R^{N}_2$ mapping to $R^{\TQ}$.  Since $\TQ_{\eta}$ is rational, $N_{\eta}$ is a smooth rational curve. Hence there is a function $f_R$ on $N_{\eta}$ such that 
$$\div(f_R)=R^N_1 -R^N_2$$
The points $R^N_1$ and $R^N_2$ are defined over some finite extension $\BL_p$ of $\K_p$.  We have $\Psi_*((N_{\eta},f_R))$ is an element of $CH^2(\TK_{\A_{\BL_p}},1)$ as, by construction 
$$\div(\Psi_*(f_R))=\Psi(R^N_1)-\Psi(R^N_2)=R^{\TQ}-R^{\TQ}=0$$


Let  $\Xi^{\Delta}_{\TK_{A_{\BL_p}}}$ be the element $\Psi_*((N_{\eta},f_R))$. One can push forward $\Xi^{\Delta}_{\TK_{A_{\BL_p}}}$ under the map $\nu$ to get an element
$$\Xi^{\Delta}_{K_{A_{\BL_p}}}=\nu_*(\Xi_{\TK_{A_{\BL_p}}})$$
 of $CH^2(K_{A_{\BL_p}},1)$ and pull it back under $\phi$ and apply the norm  to get an element  
$$\Xi^{\Delta}_A=N_{\BL_p/\K_p}(\phi^*(\Xi_{K_A})) \in CH^2(A,1).$$

\subsection{Indecomposability and the boundary.}

We would like to show that the boundary map in the localisation sequence is surjective. In the preceding section we have shown that, given a cycle in $\Pic(\A_p)$ of Humbert invariant $\Delta$  which is not the restriction of the closure of an element of $\Pic(A)$, there is an  an element $\Xi^{\Delta}_A$ in $CH^2(A,1)$. In this section we compute the boundary of this element.

 Let $D^{\Delta}_1$ and $D^{\Delta}_2$ be the  pull backs of the cycles $\QD_1$ and $\QD_2$ to $\A_p$. 

\begin{thm} The element $\Xi^{\Delta}_A$ is an indecomposable element of $CH^2(A,1)$. Further, the boundary under the map $\partial$ is the extra cycle in the special fibre $\A_p$ --
$$\partial(\Xi^{\Delta}_A)=mD^{\Delta}_1$$
for some $m \neq 0$, up to the boundary of an element of $CH_{dec}^2(A,1)$.

\label{higherchow}

\end{thm}

\begin{proof}

Recall that from Remark \ref{indres}, to show indecomposability it suffices to show the boundary of $\Xi^{\Delta}_{A}$ is a non-zero multiple of the cycle $D^{\Delta}_1$ as $D^{\Delta}_1$ does not deform to a cycle on the generic fibre -- hence cannot be the boundary of a decomposable element. We observe that we can work with the cycle  $(N^{\Delta}_{\eta},f_R)$ since computing the boundary commutes with proper push forward and so  it suffices to show that this boundary is a non-zero multiple of one of the components of its special fibre. 

Consider the curve ${\mathcal N}^{\Delta}$. Since ${\mathcal N}^{\Delta}$ is a non-singular rational curve,  the special fibre is either a non-singular curve or the union of two non-singular rational curves meeting at a point. Since ${\mathcal N}^{\Delta}$ is the normalization of $\TQQ^{\Delta}$ and  we know the special fibre of $\TQQ^{\Delta}$ consists of two curves, the special fibre of ${\mathcal N}^{\Delta}$ has to consist of two non-singular rational curves $N^{\Delta}_1$ and $N^{\Delta}_2$ meeting transversally at a point. 
$${\mathcal N}^{\Delta}_p=N^{\Delta}_1 \cup N^{\Delta}_2$$
Recall that we have an involution $\iota$ on $\TQ_1 \cup \TQ_2$ induced by the double cover.  A local calculation shows that that $\iota$ extends to the normalization $N^{\Delta}_1 \cup N^{\Delta}_2$  and one has 
$$\iota(N^{\Delta}_1)=N^{\Delta}_2.$$
\begin{lem} If $R_1$ and $R_2$ are two points of ${\mathcal N}^{\Delta}_p$ lying over a node $R$ of $\TQ_1 \cup \TQ_2$, one has 
$$\iota(R_1)=R_2.$$
and vice versa.  
\end{lem}
\begin{proof} The node $R$ is locally isomorphic  to the singularity at the origin given by $y^2=x^{2k}$. The involution $\iota$ acts by $\iota(x,y)=(x,-y)$.  In one affine part the equations defining the blow-up are $y=ux$ and $y^2=x^{2k}$ which is equivalent to  $u^2=x^{2k-2}$ and $x=0$, with the strict transform  of $y^2=x^{2k}$ given by $u^2=x^{2k-2}$.  This  is still singular if $2k-2>0$. The involution on the strict transform is given by $\iota((x,u))=(x,-u)$. In the normalisation, the equations are  $u^2=1$  and $x=0$. The involution acts the same way hence it takes the point $(0,1)$ to $(0,-1)$ and vice versa.  In our case, the points $R_1$ and $R_2$ correspond to the points $(0,1)$ and $(0,-1)$ and hence the involution interchanges them. 

\end{proof}

\begin{lem} $R^N_1$ and $R^N_2$ specialise to distinct points $R_1$ and $R_2$ which lie on different components of  the special fibre $N^{\Delta}_1 \cup N^{\Delta}_2$. \label{specialise}
\end{lem}

\begin{proof}

Suppose not. Assume both $R_1$ and $R_2$ lie on  $N^{\Delta}_1$. Then since $\Psi(R_1)=\Psi(R_2)=R$ one has  $\iota(R_1)=R_2$. Since  $\iota(N^{\Delta}_1)=N^{\Delta}_2$ , the point $R_2$ lies on $N^{\Delta}_1 \cap N^{\Delta}_2$.  So $R_2$ is a singular point. This contradicts Corollary \ref{corblowup} as $R$  is a  node satisfying the conditions of  Lemma \ref{blowup}.  
\end{proof}

Hence we can assume $R_1$ lies on $N^{\Delta}_1$ and $R_2$ lies on $N^{\Delta}_2$. We now compute $\partial ((N_{\eta}^{\Delta},f_R))$. From the definition of $\partial$ 
$$\partial ((N^{\Delta}_{\eta},f_R))=\div(\bar{f_R})=\HH+aN^{\Delta}_1+bN^{\Delta}_2$$
for some integers $a$ and $b$, where $\HH=\overline{\div (f_R)}$ is the closure of the horizontal divisor $R^N_1-R^N_2$.  A decomposable element of the form $(N_{\eta}^{\Delta},p^{k})$ has boundary 
$$\partial((N^{\Delta}_{\eta},p^k))=k (N^{\Delta}_1 + N^{\Delta}_2)$$
Hence by adding such an element with $k=-b$ to $(N^{\Delta}_{\eta},f_R)$  we may assume that $b=0$ and the boundary is of the form 
$$\partial((N^{\Delta}_{\eta},f_R))=\HH+aN^{\Delta}_ 1.$$
We now show $a \neq 0$.  From the intersection theory of arithmetic surfaces -- described in Lang \cite{lang}, chapter III, for instance -- we have that 
$$\deg( \div(\bar{f_R}|_D))=(\div(\bar{f_R}).D)=0$$
for {\em all} divisors $D$ supported in the special fibre.  Applying this with $D=N^{\Delta}_2$ we have 
$$( \HH.N^{\Delta}_2)+a(N^{\Delta}_1.N^{\Delta}_2)=0$$
From the Lemma \ref{specialise}, $\HH \cap N^{\Delta}_2=-R_2$ and so 
$$(\HH.N^{\Delta}_2)=-1.$$
From the fact that the components of the special fibre meet at precisely one point $\TP$ with multiplicity one, we  have  
$$(N^{\Delta}_1.N^{\Delta}_2)=1.$$
Hence $-1+a=0$ so $a=1$. So we have 
$$\partial((N^{\Delta}_{\eta},f_R))=\HH+N^{\Delta}_1.$$
Under the pushforward, $\Psi_*(\HH)=0$ and $\Psi_*(N^{\Delta}_1)=\QD_1$ so we have 
$$\partial(\Xi^{\Delta}_{\TK_{A_{\BL_p}}})=\QD_1$$
Combining this with the pullback to $A$ and the norm map - which could introduce a non-zero scaling factor -  we have that 
$$\partial(\Xi^{\Delta}_A)=kD^{\Delta}_1$$
 for some $k \neq 0$.  All the calculations are only up to the boundary of a decomposable  element.

\end{proof}

Returning to our original situation, from the work of Birkenhake and Wilhelm \cite{biwi} we know that $D^{\Delta}_1 \in  |\LL_0^a \otimes \LL^b_{\Delta}|$ for some $b \neq 0$ while $\D_p \in |\LD|$  so by modifying $\Xi^{\Delta}_A$ by a suitable decomposable element with can construct an element  $\Xi_{A,\D_p}$ such that 
$$\partial (\Xi_{A,\D_p})=m\D_p$$ 
for some $m \neq 0$.

\section{The Hodge-$\D$-conjecture for Abelian surfaces.}

The Hodge-$\D$-conjecture asserts that the boundary map in the localisation sequence from the higher Chow group of the generic fibre to the Chow group of the special fibre is surjective. Hence, given a cycle in the special fibre, we should be able to find a higher Chow cycle which bounds it. If the cycle in the special fibre is the restriction of a cycle in the generic fibre, then it is easily seen to be the boundary of a decomposable element of the higher Chow group.

In the previous section we have shown that, given a cycle in the special fibre which is not the restriction of a generic cycle, under certain circumstances we have constructed a higher Chow cycle which bounds it. It remains to combine this result with an analysis of the possible reductions of Abelian surfaces to get our final result.  

\begin{thm}
Let $A$ be an Abelian surface over a $p$-adic local field $\K_p$. Let $\A$ be a model over the ring of integers and $\A_p$ the special fibre. Assume $A$ has good, non-supersingular,  reduction at $p$. Then the map 
$$CH^2(A,1) \otimes \Q \stackrel{\partial}{\longrightarrow} CH^1(\A_p)\otimes \Q$$ 
is surjective. 

\end{thm}

\begin{proof} The results of the previous sections show the following  - suppose $A$, $\A$ and $\A_p$ are as before and suppose $\D_p$ is a cycle representing an element of $NS(\A_p)$ which is not the restriction of cycle in $NS(A)$. Then, from Theorem \ref{higherchow} there is a higher Chow cycle $\Xi_{A,\D_p}$ such that  
$$\partial(\Xi_{A,\D_p})=m\D_p$$
for some $m\neq 0$. 

Since we are looking at the rational Chow group, it suffices to consider the different cases up to isogeny. We observe that there are three possible cases -- depending on where $A$ or $\A_p$ is simple : 

\begin{itemize}

\item  $A$ simple and $\A_p$ simple, non-supersingular -- Here for each element $\D_p$  of $NS(\A_p)$ which is not the restriction of a generic cycle, we have a higher Chow cycle $\Xi_{A,\D_p} $ from Theorem \ref{higherchow} which bounds $\D_p$. 

\item  $A$ simple and $\A_p\simeq \E_{1,p} \times \E_{2,p}$ -- In this case, the Humbert invariant of the extra cycle $\D_p$ is $n^2$ for some $n$. If $n=1$ we can use the element constructed by Collino \cite{coll}. If $n>1$ the argument above applies. 

\item  $A \simeq E_1 \times E_2$ -- This case was covered by Spiess\cite{spie}.

\end{itemize}

\end{proof}

\begin{rem} We needed to assume non super-singular in order to use the theorem of Bogomolov-Hassett-Tschinkel. In the case of super-singular reduction, the special fibre is unirational and that allows the possibility of the cycle being deformed within the special fibre. 
\end{rem}
\begin{rem} This theorem can  be extended to work in the case when $A$ has  semi-stable reduction. A special case was studied in \cite{sree2}.

\end{rem}

\section{Some applications}

\subsection{Torsion in co-dimension two}

One immediate consequence of the construction of these higher Chow cycles is the following. Let $X$ be a smooth projective variety over a local field $K$ and assume that it has a semi-stable model over the ring of integers $\OO_K$.  Let $\Sigma_X$ be the group 
$$\Sigma_X:=\Ker(CH^2(\XX) \longrightarrow CH^2(X))$$
Then for $X$ an Abelian surface $A$ with non super-singular reduction over a $p$-adic field, $\Sigma_A$ is torsion. This is because the long exact localization sequence gives 
$$\dots \longrightarrow CH^2(A,1) \stackrel{\partial}{\longrightarrow} CH^1(\A_p) \longrightarrow CH^2(\A) \longrightarrow CH^2(A)\longrightarrow 0$$
hence the group $\Sigma_A$ is the same as the cokernel of the image of $\partial$. As we have shown that the $\partial\otimes \Q$ is surjective, this implies $\Sigma_A\otimes \Q=0$ -- hence $\Sigma_A$ is torsion. 

There are a lot of consequences of the finiteness of $\Sigma_A$ --  they are described in the paper of Spiess \cite{spie}. For example, one has 
\begin{cor}
Let $A$ be an Abelian surface over $K$ with good, ordinary, reduction. Let $\A_p$ denote its reduction mod p. Then 
\begin{itemize}
\item $CH_0(A)(\text{prime to } p) \simeq CH_0(\A_p)(\text{prime to }p)$ and in particular it is finite. 
\item The group $\Sigma_A$ is a $p$-group. 
\item For every integer $n\neq 0$ which is prime to $p$ the cycle map 
$$cl_n:CH^2(A)/n \longrightarrow H^4(A,\ZZ/n(2))$$
is injective.
\end{itemize}
\end{cor}
\begin{proof}\cite{spie} - Section 4. This is a consequence of the finiteness of $\Sigma_A$  and the fact that the Tate conjecture is known for Abelian surfaces over a finite field. 
\end{proof}

\subsection{Relations between CM cycles}

In the geometric settings the work of Mori and Mukai \cite{momu} can be used in a similar manner to deform sums of  rational curves on a particular $K3$ surface  to  rational curves on the generic fibre of the moduli of $K3$ surfaces. An argument similar to the one above can then be used to construct indecomposable  higher chow cycles in the generic Kummer surface of an Abelian surface over the Siegel modular threefold and hence can be lifted to the generic Abelian surface. 

Similar to the case studied here, these cycles degenerate over Humbert surfaces. This is a generalization of the work of Collino \cite{coll} where he constructs a higher Chow cycle which degenerates over the moduli of products of elliptic curves - namely the Humbert surface of invariant 1. In \cite{sree1}, Collino's elements were used to construct relations between CM cycles - certain codimension two cycles in the CM fibres over modular and Shimura curves. These new elements can be used to get more relations.  This in more detail will be the subject of another paper.

\subsection{$K3$ surfaces}

Since the key point of our construction was done on $K3$ surfaces and the result we used holds over fairly general $K3$ surfaces, we expect that out construction could be used to prove the non-Archimedean Hodge-$\D$-conjecture for such $K3$ surfaces. 

\bibliographystyle{alpha}
\bibliography{referencesAB}

\end{document}